\documentclass[letterpaper,12pt]{article}

\usepackage{amsmath}
\usepackage{amsfonts,amsthm,amssymb}
\usepackage{eufrak}
\usepackage{fancyhdr}
\usepackage[all]{xy}

\newtheorem{thm}{Theorem}[section]

\newtheorem{prop}[thm]{Proposition}
\newtheorem{cor}[thm]{Corollary}

\newtheorem{lem}[thm]{Lemma}

\renewcommand{\part}{\partial}
\newcommand{\deq}{\mathrel{\mathop:}=}
\newcommand{\Set}{\mathbf{Set}}

\newcommand{\sSet}{\mathbf{sSet}}

\newcommand{\slice}[2]{{{#1}/{#2}}}

\newcommand{\Ho}[1]{\mathrm{Ho}({#1})}

\newcommand{\mc}[1]{\mathcal{#1}}

\newcommand{\cosk}{\mathrm{cosk}}

\newcommand{\hocolim}[1]{{\underrightarrow{\mathop{\mathrm{holim}}}}_{#1}}

\renewcommand{\lim}{\varprojlim}
\newcommand{\colim}{\varinjlim}

\newcommand{\Shv}[1]{\mathbf{Shv}({#1})}

\newcommand{\sShv}[1]{\mathbf{s}\Shv{#1}}

\newcommand{\ets}{\acute{e}t}

\newcommand{\Spec}[1]{\textrm{Spec}\,{#1}}

\newcommand{\nsimp}{\Delta^n}

\newcommand{\Exinf}[1]{\mathrm{Ex}^{\infty}\,{#1}}
\newcommand{\Exinfn}{\mathrm{Ex}^{\infty}}

\newcommand{\ccomp}{\scriptstyle{\Pi}\displaystyle}

\newcommand{\pr}{\mathrm{pr}}

\newcommand{\hra}{\hookrightarrow}

\newcommand{\Et}{\mathrm{\acute{E}t}}

\newcommand{\HR}{\mathrm{HR}}

\newcommand{\hyp}{\mathrm{hyp}}
\newcommand{\bihyp}{\mathrm{bihyp}}

\begin{document}
\title{\'Etale Homotopy Types and Bisimplicial Hypercovers}
\author{Michael D. Misamore}
\date{August 21, 2010}
\maketitle

\begin{abstract}
The ``\'etale'' homotopy type $T(X, z)$ of any geometrically pointed locally fibrant
connected simplicial sheaf $(X, z)$ on a pointed locally connected small
Grothendieck site $(\mc{C}, x)$ is defined in terms of the geometrically pointed
hypercovers of $X$. Here this type $T(X, z)$ is compared to the analogous \'etale
homotopy type $T_b(X, z)$ constructed by means of diagonals of pointed
\emph{bisimplicial} hypercovers of $x = (X, z)$. This comparison is given by means
of the associated cocycle categories (in the sense of Jardine), and it is shown that
there are bijections
\begin{displaymath}
\pi_0 H_{\hyp}(x, y) \cong \pi_0 H_{\bihyp}(x, y)
\end{displaymath}
at the level of path components for any locally fibrant target object $y$. This
quickly leads to natural pro-isomorphisms $T(X, z) \cong T_b(X, z)$ in
$\Ho{\sSet_\ast}$, so that in particular these pro-objects in $\Ho{\sSet_\ast}$ are
weakly equivalent in the sense suggested by Artin-Mazur. By consequence one
immediately establishes the fact that $T_b(X, z)$ is invariant up to pro-isomorphism
(not just Artin-Mazur weak equivalence) under pointed local weak equivalences of
simplicial sheaves. Analogous statements for the unpointed versions of these types
also follow. The proofs given here do not in any sense rely on a pro-Whitehead type
theorem, thereby simplifying previous work in this direction.

\medbreak
\medbreak

\end{abstract}

\section{Introduction}

In classical \'etale homotopy theory, the \emph{\'etale homotopy type} of a
geometrically pointed connected locally noetherian scheme $(X, z)$ is defined by
taking objectwise connected components $\ccomp (U, u)$ of the system of all pointed
hypercovers $(U, u) \to X$ of $X$ and pointed simplicial homotopy classes of maps
between them. This system is cofiltered and thus results in a pro-object of
$\Ho{\sSet_\ast}$, the homotopy category of pointed path-connected simplicial sets.
This is the classical \'etale homotopy type of Artin-Mazur (\cite{AM}). On the other
hand, if one starts with a geometrically pointed connected locally noetherian
\emph{simplicial} scheme $X$ then one has to make a choice about what types of
hypercovers of $X$ to consider. The choice taken by Friedlander was to define his
\emph{\'etale topological type} via the system of (rigid) pointed \emph{bisimplicial}
hypercovers $(U, u) \to X$ of $X$ and fibrewise (over $X$) pointed simplicial
homotopy classes of maps between them, where a \emph{bisimplicial hypercover}
$U \to X$ is a map of bisimplicial schemes such that each degreewise map
$U_n \to X_n$ is a hypercover of the scheme $X_n$ for $n \geq 0$. There one must
take diagonals $d(U, u)$ followed by connected components in order to produce a
pro-object of $\Ho{\sSet_\ast}$, and it is obvious that this specializes to the
Artin-Mazur definition for geometrically pointed schemes $X$ regarded as constant
simplicial schemes. In \cite{Isaksen3}, Isaksen introduced another model for this
homotopy type by taking the ``realization'' of the diagram of (rigid) \'etale
homotopy types $\Et(X_n)$ of the constituent schemes, and showed in
particular that his model (also taking values in pro-simplicial sets) is weakly
equivalent to Friedlander's. Regardless of the model in pro-simplicial sets, the
\emph{homotopy type} of interest may be taken to be the pro-object $T_b(X, z)$ of
$\Ho{\sSet_\ast}$ defined by means of connected components of diagonals of
geometrically pointed bisimplicial hypercovers of $X$.

There is another possible choice for a system of geometrically pointed hypercovers
of a geomerically pointed simplicial scheme: one knows (for various reasons) that
a pointed hypercover $(U, u) \to (X, z)$ should be defined as a pointed local
trivial fibration on the relevant site, so one may simply take these instead of the
bisimplicial hypercovers. These were apparently first considered for the purpose of
defining a homotopy type by Schmidt in \cite{Schmidt1}. From this point of view it
is natural to drop the requirement that $X$ be representable by a simplicial scheme
and instead consider the system of geometrically pointed hypercovers of $X$
and pointed simplicial homotopy classes of maps between them for any geometrically
pointed connected locally fibrant simplicial sheaf $(X, z)$, where the words
``geometrically pointed'' are suitably defined. One way or another, it has been
known for some time that this also results in a pro-object $T(X, z)$ of
$\Ho{\sSet_\ast}$; the underlying ideas go back to Brown's thesis \cite{Brown1}.

This raises the fundamental question of how to compare $T(X, z)$ with the analogous
type $T_b(X, z)$ defined above. Here this comparison is achieved by working at the
level of \emph{cocycle categories} in the sense of \cite{Jardine4}: it is shown here
(Theorem \ref{thmbijs}) that there are bijections
\begin{displaymath}
\pi_0 H_{\hyp}(x, y) \cong \pi_0 H_{\bihyp}(x, y) \cong [x, y]
\end{displaymath}
between the path components of the cocycle categories for ordinary and bisimplicial
hypercovers for any locally fibrant pointed simplicial sheaf $y$ on the ambient
small Grothendieck site $\mc{C}$. For these bijections one only requires that the
site $\mc{C}$ be pointed (and only if one desires to speak about \emph{pointed}
hypercovers). As these results really have no dependence on the \'etale topology
per se, they may be of general interest.

Different variants of abelian and nonabelian sheaf cohomology may then be recovered
from bisimplicial hypercovers by means of generalized Verdier hypercovering
arguments; for the sake of example, the identifications
\begin{displaymath}
H^n(X, H) \cong \colim_{p:\, d(U) \to X}{H^n(dU, H)}
\end{displaymath}
for sheaves of groups $H$ are established in Proposition \ref{verd1}. These results
are proven without spectral sequence arguments and work equally well for nonabelian
$H^1$.

Finally, it is shown here in Theorem \ref{mainthm} that $T_b(X, z)$ is indeed
pro-isomorphic to $T(X, z)$ whenever $(X, z)$ is a pointed connected locally fibrant
simplicial sheaf on a pointed locally connected site where the distinguished
``point'' is determined by some object $\Omega$ representing a sheaf (such as a
geometric point). The resulting invariance of $T_b(X, z)$ up to pro-isomorphism
under (pointed) local weak equivalences is the subject of Corollary \ref{maincor}.
The $\Exinfn$ functor is employed in Lemma \ref{exinf} to demonstrate that this
invariance holds without any fibrancy assumptions on $(X, z)$. One recovers in
particular the fact that a bisimplicial hypercover $U \to X$ determines a
pro-isomorphism $T_b(dU) \cong T_b(X)$ of bisimplicial \'etale homotopy types
(cf. $8.1$, \cite{Friedlander1}); the proof here is quite elementary and does not
make use of the pro-Whitehead Theorem from \S$4$, \cite{AM}.

\section{Hypercovers and bisimplicial hypercovers} \label{secthyp}

A map $U \to X$ of simplicial (pre)sheaves on a small Grothendieck site $\mc{C}$ is
called a \emph{hypercover} if it is a local fibration and a local weak equivalence
(cf. \cite{Jardine9} for a definition and discussion of the local right lifting
property defining local fibrations), and it is well known that when $X = K(X, 0)$
is the discrete or ``constant'' simplicial (pre)sheaf associated to an object $X$
of $\mc{C}$, the map $U \to X$ as above is a hypercover in this sense exactly when
the maps
\begin{eqnarray*}
U_0 & \to & X_0 \\
U_n & \to & \cosk_{n-1}{U_n}
\end{eqnarray*}
are local epimorphisms of (pre)sheaves on $\mc{C}$ for $n \geq 1$, which may be
taken as the ``classical'' definition. The fact that these definitions correspond
follows from ($1.12$, \cite{Jardine3}) in the case where $X$ is locally fibrant, or
by a Boolean localization argument in the general case, following Jardine
\cite{Jardine16}.

Suppose $\mc{C}$ is \emph{pointed} in the sense that there is a geometric morphism
\begin{displaymath}
x: \Set \to \Shv{\mc{C}}
\end{displaymath}
of toposes. For the purposes of this paper, a \emph{pointed simplicial sheaf}
$(X, z)$ on $\mc{C}$ will be a simplicial sheaf $X$ on $\mc{C}$ together with a
choice of section $z \in x^\ast(X_0)$, and a \emph{pointed map} $f: (X, z) \to
(Y, z^\prime)$ of pointed simplicial sheaves will be a map $X \to Y$ of the
underlying simplicial sheaves such that $x^\ast(f)(z) = z^\prime$. In the usual
geometric setting for \'etale homotopy theory, such ``points'' $z$ correspond to
geometric points of $X$ whenever $X = K(X, 0)$ is a discrete representable
simplicial sheaf. A \emph{pointed hypercover} $(U, u) \to (X, z)$ of a pointed
simplicial sheaf $(X, z)$ on a pointed small Grothendieck site $(\mc{C}, x)$ will
be a hypercover $U \to X$ that is a pointed map (with respect to $x$) of simplicial
sheaves, and a \emph{pointed map} of pointed hypercovers of $(X, z)$ will be a
pointed map over $X$ of the underlying simplicial sheaves. It has been observed
as early as \cite{Brown1} that the pointed hypercovers of any \emph{locally fibrant}
pointed simplicial sheaf $(X, z)$ together with pointed simplicial homotopy classes
of maps between them over $(X, z)$ form a cofiltered category, here denoted
$\HR_\ast(X, z)$.

A \emph{bisimplicial hypercover} $f: U \to X$ of a simplicial (pre)sheaf $X$ on a
small Grothendieck site $\mc{C}$ is a map of bisimplicial (pre)sheaves
$f: U \to K_v(X, 0)$, where $X$ is being regarded as simplicially discrete in the
``vertical'' direction, such that each of the constituent maps $f_m: U_m \to X_m$
in ``horizontal'' degree $m \geq 0$ is a hypercover. If $(\mc{C}, x)$ is a pointed
site then a \emph{pointed bisimplicial hypercover} $f: (U, u) \to (X, z)$ of a 
pointed simplicial sheaf $(X, z)$ is a bisimplicial hypercover $f: U \to X$ in
sheaves such that $x^\ast(f)(u) = z$, where $u \in x^\ast(U_{0,0})$ is the ``point''
associated to $U$. A \emph{pointed map} $f: (V, v) \to (U, u)$ of pointed
bisimplicial hypercovers of a pointed simplicial sheaf $(X, z)$ is a map of
bisimplicial hypercovers $f: V \to U$ over $X$ such that $x^\ast(f)(v) = u$. The
most significant fact about these objects for the present purposes is given by

\begin{prop} \label{diagwe}
If $f: (U, u) \to (X, z)$ is a bisimplicial hypercover of a pointed simplicial sheaf
$(X, z)$ on a pointed small Grothendieck site $(\mc{C}, x)$ then the map $(dU, u)
\to (X, z)$ of simplicial sheaves induced by taking diagonals is a local weak
equivalence.
\end{prop}
\begin{proof}
Fixing a Boolean localization $p: \Shv{\mc{B}} \to \Shv{\mc{C}}$, it suffices to
show that the induced map $p^\ast(dU) \to p^\ast(X)$ is a sectionwise weak
equivalence by techniques of \cite{Jardine16}, but this follows from the
corresponding fact for simplicial sets ($1.7$, IV, \cite{GJ}). The pointedness of
the induced map is trivial.
\end{proof}
This Proposition and ($8.1$, \cite{Friedlander1}) serve as motivation for asking
whether \emph{any} local weak equivalence induces isomorphisms on the \'etale
homotopy pro-groups of the associated bisimplicial \'etale homotopy types. Reader
beware: the diagonal of a pointwise fibration of simplicial sets need \emph{not} be
a fibration in general, so one does not generally expect the diagonal of a
bisimplicial hypercover to be a hypercover. This is a source of technical problems
when it comes to comparing the differing definitions of \'etale homotopy types.

\section{The \'Etale Homotopy Type of a Simplicial Sheaf}

Say that a small Grothendieck site $\mc{C}$ is \emph{locally connected} if there
exists a functor $\ccomp: \Shv{\mc{C}} \to \Set$ left adjoint to the constant sheaf
functor $\Gamma^\ast: \Set \to \Shv{\mc{C}}$, and say that a locally connected site
$\mc{C}$ is \emph{connected} if $\ccomp(\ast) = \ast$ where $\ast$ denotes the
terminal sheaf on $\mc{C}$. In geometric situations the functor $\ccomp$ is that
induced by the functor which sends any scheme to its set of connected
scheme-theoretic components. A simplicial sheaf $X$ on a connected site $\mc{C}$
will be called \emph{connected} if $\pi_0 \ccomp(X) \cong \ast$; a quick argument
using $H^0(-, K(\Gamma^\ast S, 0))$ for variable sets $S$ shows that $U$ is
connected whenever $U \to X$ is a hypercover of a connected simplicial sheaf $X$,
and a similar statement is true for diagonals of bisimplicial hypercovers by the
same argument.

Suppose $(\mc{C}, x)$ is a pointed locally connected small Grothendieck site and
that $(X, z)$ is a pointed (with respect to $x$) connected locally fibrant
simplicial sheaf on $\mc{C}$. Then the pointed hypercovers of $(X, z)$ are
cofiltered up to simplicial homotopy so one may proceed to define an ``\'etale''
\emph{homotopy type} $T(X, z)$ for $(X, z)$: it is the pro-object in
$\Ho{\sSet_\ast}$ given by applying $\ccomp$ to the cofiltered diagram
$\HR_\ast(X, z)$ of pointed hypercovers of $(X, z)$ and pointed simplicial homotopy
classes of maps between them. This definition applies in particular to pointed
locally fibrant connected simplicial schemes $(X, z)$ on \'etale sites, and is
clearly \emph{not} the same as the \'etale topological type of Friedlander defined
by means of diagonals of (rigid) pointed bisimplicial hypercovers of $(X, z)$.
Nevertheless $T(X, z)$ has several good properties: firstly, it specializes to
the classical \'etale homotopy type for geometrically pointed connected schemes
(this is a matter of checking definitions). Secondly, the fact that it is defined
in terms of not-necessarily-representable hypercovers does not matter:

\begin{prop} \label{compclass}
For any geometrically pointed connected scheme $(X, z)$ on a locally connected
\'etale site, the \'etale homotopy type $T(X, z)$ defined here is pro-isomorphic
to the classical \'etale homotopy type $\Et(X, z)$ of Artin-Mazur defined by means
of pointed \emph{representable} hypercovers of $(X, z)$.
\end{prop}
\begin{proof}
The only point is to show that any pointed hypercover $(U, u) \to (X, z)$ of a
scheme $X$ can be refined by a pointed representable hypercover (as then the result
follows by a cofinality argument in $\HR_\ast(X, z)$). This construction was given
by Jardine in \cite{Jardine8} based on the work of Artin-Mazur (\S$8$, \cite{AM}).
\end{proof}

Next, one may show directly that the type $T(U, u)$ is pro-isomorphic to the type
$T(X, z)$ for any pointed hypercover $(U, u) \to (X, z)$ of $(X, z)$. The
corresponding (actually weaker) statement for the \'etale topological type of
Friedlander requires some work to establish (cf. $8.1$, \cite{Friedlander1}) but
is easy to prove for $T(-, -)$:
\begin{lem} \label{lemhyp}
Suppose $(\mc{C}, x)$ is a pointed locally connected small Grothendieck site,
$(X, z)$ a pointed connected locally fibrant simplicial sheaf on $\mc{C}$, and
$f: (U, u) \to (X, z)$ a pointed hypercover of $(X, z)$. Then $f$ induces a
pro-isomorphism $T(U, u) \cong T(X, z)$.
\end{lem}
\begin{proof}
Consider the slice category $\slice{\HR_\ast(X, z)}{f}$ whose objects are the
commutative triangles
\begin{displaymath}
\xymatrix{ (V, v) \ar[rr]^g \ar[dr]_h &  & (U, u) \ar[dl]^f \\
           &  (X, z)  &  }
\end{displaymath}
over $(X, z)$ where $h$ is a hypercover. The maps $g$ may not be hypercovers
themselves, but as any such $V$ and $U$ are locally fibrant (since $X$ is), any
such object has a functorial refinement up to weak equivalence by an object
$g^\prime: (Z, z) \to (U, u)$ over $(X, z)$ such that $g^\prime$ is a pointed
hypercover of $(U, u)$: this is determined by the usual factorization
\begin{displaymath}
\xymatrix{  &  (Z, z) \ar@/_2ex/[dl]_k \ar[dr]^{g^\prime} & \\
          (V, v) \ar[rr]^g \ar[ur]_w  &  &  (U, u) }
\end{displaymath}
where $g^\prime$ is a local fibration (and local weak equivalence by closed model
axiom CM$2$) and $w$ is a local weak equivalence that is a section of the local
trivial fibration $k$. This determines a full cofinal subcategory
\begin{displaymath}
i: (\slice{\HR_\ast(X)}{f})_{\mathrm{hyp}} \hra \slice{\HR_\ast(X)}{f}
\end{displaymath}
as one can see by equalizing any two maps of pointed hypercovers over $X$ up to
pointed simplicial homotopy and replacing the resulting equalizer
$(E, e) \to (U, u)$ over $X$ by a hypercover $(E^\prime, e^\prime) \to (U, u)$ by
factorization as above. There is an equivalence $f_\ast: \HR_{\ast}(U) \simeq
(\slice{\HR_\ast(X)}{f})_{\mathrm{hyp}}$ defined by composing with $f$ or likewise
forgetting the maps to $X$. Further, the functor $p: \slice{\HR_\ast(X)}{f} \to
\HR_\ast(X)$ defined by forgetting the maps to $f$ is cofinal since $\HR_\ast(X)$
is cofiltered. It follows that the composite
\begin{displaymath}
\HR_\ast(U) \xrightarrow{f_\ast} (\slice{\HR_\ast(X)}{f})_{\mathrm{hyp}}
   \xrightarrow{pi} \HR_{\ast}(X)
\end{displaymath}
is cofinal, and this composite induces the desired pro-isomorphism.
\end{proof}

This cofinality argument works essentially because $T(-, -)$ treats both the base
simplicial sheaf $X$ and its pointed hypercovers on the same footing; such an
argument therefore fails for \emph{bisimplicial} hypercovers of simplicial sheaves.
It almost immediately follows that $T(X, z)$ is invariant up to pro-isomorphism
under pointed local weak equivalences:
\begin{prop} \label{invwe}
Suppose $(\mc{C}, x)$ is a pointed locally connected small Grothendieck site,
$(X, z)$ and $(Y, y)$ pointed connected locally fibrant simplicial sheaves on
$\mc{C}$, and $f: (Y, y) \to (X, z)$ a pointed local weak equivalence. Then $f$
induces a pro-isomorphism $T(Y, y) \cong T(X, z)$.
\end{prop}
\begin{proof}
As $Y$ and $X$ are locally fibrant there is a factorization
\begin{displaymath}
\xymatrix{  &  (W, w) \ar@/_2ex/[dl]_g \ar[dr]^h & \\
           (Y, y) \ar[ur]_i \ar[rr]^f  &  &  (X, z) }
\end{displaymath}
of the map $f$ where $g$ and $h$ are pointed hypercovers and $i$ is right inverse
to $g$. The induced map $i_\ast: T(Y, y) \to T(W, w)$ is right inverse to the
induced map $g_\ast$, which is a pro-isomorphism by Lemma \ref{lemhyp}, so $i_\ast$
is also left inverse to $g_\ast$ and a pro-isomorphism. The map $h_\ast$ is also a
pro-isomorphism by the Lemma so $h_\ast i_\ast = f_\ast: T(Y, y) \to T(X, z)$ is a
pro-isomorphism.
\end{proof}

In an earlier work (\cite{Schmidt1}), Schmidt observed that this latter fact is
implied by an application of the generalized Verdier hypercovering theorem; the
present proof is included because it may be of independent interest. To give
another comparison with known results, recall that in ($2$, \cite{Isaksen3})
Isaksen shows this his \'etale realization functor $\mathrm{Re}^{\ets}$ is left
Quillen for the local projective structure so that it sends local weak
equivalences between local projective cofibrant simplicial presheaves to weak
equivalences of pro-simplicial sets. The type $T(-, -)$ defined here sends
pointed local weak equivalences between pointed connected locally fibrant
simplicial sheaves to pro-isomorphisms in $\Ho{\sSet_\ast}$. Obviously the
points are only there to make a comparison of homotopy pro-groups: this general
line of argument continues to work in the unpointed case.

\section{Cocycles for bisimplicial sheaves}

In what follows $\mc{C}$ will be an arbitrary small Grothendieck site and
$\Set \to \Shv{\mc{C}}$ a point of $\mc{C}$ (or no point at all in the unpointed
situation). Familiarity with the definitions of \cite{Jardine4} and \cite{Jardine13}
will be assumed. The word ``pointed'' will always mean with respect to the chosen
point of $\mc{C}$ rather than with respect to the terminal sheaf. Pointed
(bi)simplicial sheaves on $\mc{C}$ will be denoted from now on simply with lowercase
letters to reduce the notational burden, and unless otherwise specified any map
$x \to y$ between pointed (bi)simplicial sheaves will be pointed. The underlying
site will always be $\mc{C}$.

For any two pointed simplicial sheaves $x$, $z$ on $\mc{C}$ there is a category
$H_{\bihyp}(x, z)$ of cocycles of the form
\begin{displaymath}
x \xleftarrow{d(f)\, \simeq} d(u) \xrightarrow{p} z
\end{displaymath}
where $f: u \to x$ is a pointed bisimplicial hypercover of $x$ and $d$ is the
diagonal functor, whose morphisms are commutative diagrams
\begin{displaymath}
\xymatrix@=20pt{  &  d(u) \ar[dl]_{d(f)} \ar[dr]^p \ar[dd]^{d(m)} &  \\
           x &   &  z \\
	    &  d(u^\prime) \ar[ul]^{d(f^\prime)} \ar[ur]_{p^\prime} & }
\end{displaymath}
where $m: u \to u^\prime$ is any pointed map of bisimplicial hypercovers of $x$.

As these categories turn out to be a bit tricky to study directly, one may also
consider categories $H_d(x, z)$ whose objects are cocycles of the form
\begin{displaymath}
x \xleftarrow{d(f)\, \simeq} d(u) \xrightarrow{p} z
\end{displaymath}
where $f: u \to x$ is any pointed map of bisimplicial sheaves that is a
\emph{diagonal local weak equivalence} in the sense that $d(f)$ is a local weak
equivalence of simplicial sheaves, and whose morphisms are similarly defined; it
is immediate from the definition that $H_{\bihyp}(x, z)$ is a full subcategory
of $H_d(x, z)$ for any fixed $x$ and $z$.

Recall the Moerdijk closed model structure for bisimplicial sets: the
\emph{fibrations} (resp. \emph{weak equivalences}) are by definition the diagonal
fibrations (resp. diagonal weak equivalences), and the cofibrations are defined
by the left lifting property with respect to all trivial fibrations
(cf. $3.15$, IV, \cite{GJ}). Every Moerdijk cofibration is a monomorphism of
bisimplicial sets and therefore a diagonal cofibration in particular.

The diagonal functor $d$ has a right adjoint $d_\ast$ ($3.13$, IV, \cite{GJ}) so
that any object
\begin{displaymath}
x \xleftarrow{d(f)\, \simeq} d(u) \xrightarrow{p} z
\end{displaymath}
of $H_d(x, z)$ is uniquely identified with a diagram
\begin{displaymath}
x \xleftarrow{f} u \xrightarrow{\tilde{p}} d_\ast(z)
\end{displaymath}
where $f$ is the underlying map of bisimplicial sheaves and $\tilde{p}$ is the
adjoint of $p$. In other words any such cocycle $(d(f), p)$ may be identified with
a ``cocycle'' $(f, \tilde{p})$.

\begin{lem} \label{mordfib}
If $f: X \to Y$ is a trivial fibration of simplicial sets then the induced map
$d_\ast(f): d_\ast X \to d_\ast Y$ is a trivial fibration for the Moerdijk closed
model structure on bisimplicial sets.
\end{lem}
\begin{proof}
It must be shown that $d_\ast(f)$ has the right lifting property with respect to any
Moerdijk cofibration $i: A \to B$. By adjointness, this is equivalent to showing
that $f$ has the right lifting property with respect to $d(i)$, but $d(i)$ is a
cofibration of simplicial sets since $i: A \to B$ is in particular a monomorphism of
bisimplicial sets ($3.15$, IV, \cite{GJ}). The required lift exists since $f$ is a
trivial fibration.
\end{proof}

\begin{cor} \label{diagequiv}
If $f: X \to Y$ is a weak equivalence of fibrant simplicial sets then the induced
map $d_\ast(f): d_\ast X \to d_\ast Y$ is a diagonal weak equivalence of
bisimplicial sets.
\end{cor}
\begin{proof}
Factor the map $f$ as a weak equivalence $\sigma$ followed by trivial fibration $g$
such that $\sigma$ is a section of a trivial fibration $h$. Then $d_\ast(h)$ and
$d_\ast(g)$ are trivial fibrations for the Moerdijk structure by Lemma
\ref{mordfib}, so $d_\ast(f)$ is a diagonal equivalence.
\end{proof}

The ``localized'' version of this is then given by
\begin{lem} \label{locdequiv}
If $\beta: z \to z^\prime$ is a local weak equivalence of pointed locally fibrant
simplicial sheaves then $d_\ast(\beta)$ is a pointed diagonal local weak
equivalence.
\end{lem}
\begin{proof}
Fix a Boolean localization $p: \Shv{\mc{B}} \to \Shv{\mc{C}}$. Then the map
$p^\ast(\beta)$ is a sectionwise weak equivalence of sectionwise fibrant simplicial
sheaves on $\mc{B}$ so that $d_\ast(p^\ast(\beta))$ is a sectionwise diagonal
weak equivalence of bisimplicial sheaves by Lemma \ref{diagequiv}. In bisimplicial
degree $(m, n)$ this map is given by the sheaf map
$(p^\ast \beta)^{\Delta^m \times \nsimp}$ where $\Delta^m \times \nsimp$ is the
constant sheaf associated to the corresponding simplicial set. As $p^\ast$ is exact
one has $p^\ast(\Delta^m \times \nsimp) = \Delta^m \times \nsimp$ so that the
map $(p^\ast \beta)^{\Delta^m \times \nsimp}$ is isomorphic to the map
$p^\ast(\beta^{\Delta^m \times \nsimp})$ naturally in $m$, $n \geq 0$. Thus the
map $p^\ast(d_\ast(\beta))$ is also a sectionwise diagonal weak equivalence.
By exactness of $d$ and $p^\ast$ it follows that $p^\ast(d(d_\ast(\beta)))$ is a
sectionwise weak equivalence so that $d(d_\ast(\beta))$ must be a local weak
equivalence, hence $d_\ast(\beta)$ is a diagonal local weak equivalence as was
to be shown.
\end{proof}

Here is a first application to cocycles:
\begin{lem} \label{lwed}
If $\beta: z \to z^\prime$ is a local weak equivalence of pointed locally fibrant
simplicial sheaves then the induced functor
\begin{displaymath}
\beta_\ast: BH_d(x, z) \to BH_d(x, z^\prime)
\end{displaymath}
is a homotopy equivalence.
\end{lem}
\begin{proof}
The functor $\beta_\ast$ is defined by composition with $\beta$, sending the
righthand map from $p$ to $p^\prime \deq \beta p$. The local weak equivalence
$\beta$ determines a diagonal local weak equivalence $d_\ast(\beta)$ by Lemma
\ref{locdequiv}, and by adjunction the associated ``cocycle'' $(f, \tilde{p})$ is
sent to $\tilde{p}^\prime \deq d_\ast(\beta)\tilde{p}$.

Suppose one has a pointed ``cocycle'' of the form
\begin{displaymath}
x \xleftarrow{f} u \xrightarrow{\tilde{p}} d_\ast(z^\prime)
\end{displaymath}
where $f$ is a diagonal local weak equivalence. Emulating the proof of Lemma $1$
of \cite{Jardine4}, this is equivalent to giving a map
$(f, \tilde{p}): u \to x \times d_\ast(z^\prime)$ which may be factored using the
Moerdijk structure in sections as
\begin{displaymath}
\xymatrix{  &  w \ar[dr]^{(p_x, \tilde{p}_{z^\prime})}   & \\
           u \ar[ur]^c \ar[rr]^(0.4){(f, \tilde{p})}  &  & x \times d_\ast(z^\prime)}
\end{displaymath}
where $c$ is a sectionwise diagonal equivalence and $(p_x, \tilde{p}_{z^\prime})$
is a sectionwise diagonal fibration. Observe that $p_x$ is a diagonal local weak
equivalence since both $f$ and $c$ are. Pull back along the diagonal equivalence
$1 \times d_\ast(\beta)$ to get a commutative square
\begin{displaymath}
\xymatrix@=40pt{ w^\prime \ar[r]^{(1 \times d_\ast(\beta))_\ast}
                   \ar[d]_{(p_x^\ast, \tilde{p}_z^\ast)} &
                 w \ar[d]^{(p_x, \tilde{p}_{z^\prime})} \\
                 x \times d_\ast(z) \ar[r]_{1 \times d_\ast(\beta)}  &
                 x \times d_\ast(z^\prime) }
\end{displaymath}
defining $w^\prime$. The map $(p_x^\ast, \tilde{p}_z^\ast)$ is a sectionwise
diagonal fibration since $(p_x, \tilde{p}_{z^\prime})$ was, and passing to a
Boolean localization $p^\ast$ preserves pullbacks (and $d$ is exact) so the map
$(1 \times d_\ast(\beta))_\ast$ is a diagonal local weak equivalence and therefore
so is the map $p_x^\ast$. This determines a functor $\tilde{\psi}$ from the
category of pointed ``cocycles'' of the form
\begin{displaymath}
x \xleftarrow{f} u \xrightarrow{\tilde{p}} d_\ast(z^\prime)
\end{displaymath}
to the analogous category of such objects with target $d_\ast(z)$. The canonical
maps
\begin{displaymath}
(f, \tilde{p}) \to (p_x, \tilde{p}_{z^\prime}) \leftarrow \beta_\ast
  \tilde{\psi} (f, \tilde{p})
\end{displaymath}
and
\begin{displaymath}
(g, \tilde{q}) \to \tilde{\psi}\beta_\ast(g, \tilde{q})
\end{displaymath}
determine natural transformations (use that $w^\prime$ is a pullback for the
latter). The aforementioned categories are therefore homotopy equivalent. These
categories are isomorphic to $H_d(x, z)$ and $H_d(x, z^\prime)$, respectively, so
the result follows.
\end{proof}

\begin{cor}
If $\beta: y \to z$ is a globally fibrant model of a pointed locally fibrant
simplicial sheaf $y$ then the induced map
\begin{displaymath}
\pi_0(\beta_\ast): \pi_0 H_d(x, y) \to \pi_0 H_d(x, z)
\end{displaymath}
is a bijection.
\end{cor}
\begin{proof}
Globally fibrant objects are locally fibrant. Apply Lemma \ref{lwed}.
\end{proof}

In Lemma $4$ of \cite{Jardine13}, Jardine established that if
\begin{displaymath}
x \xleftarrow{f\, \simeq} u \to z
\end{displaymath}
is any cocycle of pointed simplicial (pre)sheaves with $z$ locally fibrant then it
may be functorially replaced by a cocycle of the form
\begin{displaymath}
x \xleftarrow{f^\prime\,\simeq} u^\prime \to z
\end{displaymath}
where $f^\prime$ is a hypercover, such that the new cocycle is in the same path
component of $H(x, z)$ as the original. It follows that the inclusion functor
\begin{displaymath}
j: H_{\hyp}(x, z) \hra H(x, z)
\end{displaymath}
of the full subcategory $H_{\hyp}(x, z)$ into $H(x, z)$ induces a bijection on path
components for any locally fibrant simplicial (pre)sheaf $z$.

To give an analogue for $H_d(x, z)$, let $H_{d-\hyp}(x, z)$ denote the category of
cocycles of the form
\begin{displaymath}
x \xleftarrow{d(f)\, \simeq} d(u) \to z
\end{displaymath}
where $f: u \to x$ is any map of pointed bisimplicial sheaves to a pointed
simplicial sheaf $x$ such that the induced map $d(f)$ is a hypercover, whose
morphisms are given by maps $m: u \to u^\prime$ inducing maps of such cocycles in
the usual way. Then $H_{d-\hyp}(x, z)$ is another full subcategory of $H_d(x, z)$
so that the inclusion functor
\begin{displaymath}
i^\prime: H_{d-\hyp}(x, z) \hra H_d(x, z)
\end{displaymath}
is injective on path components.

Say that a map $f: x \to y$ of pointed bisimplicial sheaves is a \emph{diagonal
local fibration} if the induced map $d(f)$ is a local fibration of simplicial
sheaves.

\begin{lem}
If $z$ is a pointed locally fibrant simplicial sheaf on $\mc{C}$ then $d_\ast(z)
\to \ast$ is a diagonal local fibration.
\end{lem}
\begin{proof}
Fix a Boolean localization $p: \Shv{\mc{B}} \to \Shv{\mc{C}}$. The map
$p^\ast(z) \to \ast$ is a sectionwise fibration since $z$ is locally fibrant, so
the induced map $d_\ast(p^\ast(z)) \to \ast$ is a sectionwise diagonal fibration
by ($3.14$, IV, \cite{GJ}). By the proof of Lemma \ref{locdequiv} this implies
that the map $p^\ast(d_\ast(z)) \to \ast$ is also a sectionwise diagonal fibration,
but by exactness this implies that $p^\ast(d(d_\ast(z))) \to \ast$ is a sectionwise
fibration so that $d(d_\ast(z)) \to \ast$ is a local fibration, thus
$d_\ast(z) \to \ast$ is a diagonal local fibration.
\end{proof}

\begin{lem}
If $z$ is locally fibrant then the induced map $\pi_0(i^\prime)$ is a bijection.
\end{lem}
\begin{proof}
As above, identify any object
\begin{displaymath}
x \xleftarrow{d(f)} u \xrightarrow{p} z
\end{displaymath}
of $H_d(x, z)$ with the corresponding map
\begin{displaymath}
(f, \tilde{p}): u \to x \times d_\ast(z)
\end{displaymath}
and factor $(f, \tilde{p})$ as a sectionwise trivial Moerdijk cofibration
$c: u \to w$ followed by a sectionwise Moerdijk fibration
$(p_x, \tilde{p}_z): w \to x \times d_\ast(z)$. Then $c$ and $f$ are diagonal local
weak equivalences so $p_x$ is a diagonal local weak equivalence. The map $p_x$ is
the composite
\begin{displaymath}
w \xrightarrow{(p_x, p_z)} x \times d_\ast(z) \xrightarrow{\pr_x} x
\end{displaymath}
where the first map is a sectionwise diagonal fibration and the second map is a
diagonal local fibration since $d_\ast(z) \to \ast$ is a diagonal local fibration.
Thus $d(p_x)$ is a hypercover and the result follows by adjunction.
\end{proof}

The category $H_{d-\hyp}(x, z)$ is a (\emph{not} a priori full) subcategory of
$H_{\hyp}(x, z)$, the full subcategory of $H(x, z)$ whose objects are cocycles of
the form
\begin{displaymath}
x \xleftarrow{f\, \simeq} v \to z
\end{displaymath}
where $f$ is any hypercover. Let
\begin{displaymath}
i^{\prime\prime}: H_{d-\hyp}(x, z) \hra H_{\hyp}(x, z)
\end{displaymath}
denote the inclusion.

\begin{lem}
For any pointed locally fibrant simplicial sheaf $z$ the induced map
$\pi_0(i^{\prime\prime})$ is surjective.
\end{lem}
\begin{proof}
Let
\begin{displaymath}
x \xleftarrow{f\, \simeq} v \xrightarrow{p} z
\end{displaymath}
be any object of $H_{\hyp}(x, z)$ and identify it with the map 
$(f, p): v \to x \times z$. Factor $(f, p)$ as a sectionwise trivial Moerdijk
cofibration $c: v \to w$ followed by a sectionwise Moerdijk fibration
$(p_x, p_z): w \to x \times z$. The maps $f$ and $c$ are diagonal local weak
equivalences so $p_x$ is a diagonal local weak equivalence. The map $p_x$ is the
composite
\begin{displaymath}
w \xrightarrow{(p_x, p_z)} x \times z \xrightarrow{\pr_x} x
\end{displaymath}
where the first map is a sectionwise diagonal fibration and the second map is a
diagonal local fibration since $z$ is locally fibrant. Thus $d(p_x)$ is a
hypercover, so the result follows.
\end{proof}

Recall that the diagonal functor $d$ also has a left adjoint $d^\ast$ ($3.3$, IV,
\cite{GJ}). To show injectivity of $\pi_0(i^{\prime\prime})$ one may use a
roundabout argument beginning with
\begin{lem} \label{unitmap}
For any simplicial set $X$, the unit map $\eta: X \to dd^\ast(X)$ is a weak
equivalence.
\end{lem}
\begin{proof}
Any Moerdijk cofibration $c$ is a monomorphism in particular ($3.15$, IV,
\cite{GJ}), so $d(c)$ is a monomorphism and hence a cofibration of simplicial sets.
If $e$ is any Moerdijk weak equivalence then $d(e)$ is a weak equivalence of
simplicial sets by definition, thus $d$ preserves cofibrations, weak equivalences,
and colimits (it is left adjoint to a functor $d_\ast$: cf. \cite{GJ}). Its right
adjoint $d_\ast$ preserves trivial fibrations, hence pointwise trivial fibrations
of diagrams of simplicial sets.

The left adjoint $d^\ast$ of $d$ preserves colimits by definition, and if
$c: A \hra B$ is any trivial cofibration of simplicial sets then $d^\ast(c)$ is a
trivial Moerdijk cofibration by adjointness and CM$4$ for the Moerdijk closed
model structure. Any weak equivalence $e: X \to Y$ of simplicial sets factors as
\begin{displaymath}
\xymatrix{ & Z \ar[dr]^p & \\
           X \ar[rr]^e \ar[ur]^c &  &  Y \ar@/_3ex/[ul]_{c^\prime} }
\end{displaymath}
where $c$ is a trivial cofibration and $p$ is left inverse to a trivial cofibration
$c^\prime$, thus $d^\ast$ sends weak equivalences of simplicial sets to Moerdijk
weak equivalences (also, $d^\ast$ sends cofibrations to Moerdijk cofibrations by
another adjointness plus CM$4$ argument). The right adjoint $d$ of $d^\ast$
preserves trivial fibrations by definition of the Moerdijk structure, so preserves
pointwise trivial fibrations of diagrams of bisimplicial sets.

The composite $dd_\ast$ therefore preserves pointwise trivial fibrations of diagrams
of simplicial sets, so its left adjoint $dd^\ast$ preserves projective cofibrations.
As $d$ and $d^\ast$ both preserve weak equivalences and colimits, $dd^\ast$
preserves projective cofibrant models of diagrams of simplicial sets and therefore
homotopy colimits.

Observe that
\begin{displaymath}
dd^\ast(\nsimp) = d\Delta^{n,n} = \nsimp \times \nsimp \simeq \ast
\end{displaymath}
for $n \geq 0$ so that the canonical maps $\eta: \nsimp \to dd^\ast(\nsimp)$ are
weak equivalences for $n \geq 0$. For any simplicial set $X$ there is a canonical
weak equivalence
\begin{displaymath}
\hocolim{\slice{\Delta}{X}}{\nsimp} \simeq X
\end{displaymath}
(cf. $5.2$, IV, \cite{GJ}). Consider the commutative square
\begin{displaymath}
\xymatrix{ \hocolim{\slice{\Delta}{X}}{\nsimp} \ar[r]^(0.6){\simeq}
              \ar[d]_\eta &  X \ar[d]^\eta \\
           dd^\ast\hocolim{\slice{\Delta}{X}}{\nsimp} \ar[r]^(0.6)\simeq &
              dd^\ast(X) }
\end{displaymath}
The top and bottom maps are weak equivalences and the lefthand map is a weak
equivalence as it is weakly equivalent to the map
$\hocolim{\slice{\Delta}{X}}{\eta_n}$ where the
\begin{displaymath}
\eta_n: \nsimp \to dd^\ast(\nsimp)
\end{displaymath}
are the unit maps for $(d^\ast, d)$ applied to $\nsimp$ for $n \geq 0$.
\end{proof}

\begin{cor}
The adjunction $(d^\ast, d)$ is a Quillen equivalence between the standard closed
model structure on simplicial sets and the Moerdijk closed model structure on
bisimplicial sets.
\end{cor}
\begin{proof}
The functor $d^\ast$ preserves cofibrations and weak equivalences by the proof of
the previous Lemma, and $d$ preserves fibrations and weak equivalences by definition
of the Moerdijk closed model structure, hence $(d^\ast, d)$ is a Quillen adjunction.
If $f: X \to dY$ is a weak equivalence of simplicial sets, the adjoint map
$\tilde{f}: d^\ast X \to Y$ defined by the adjunction diagram
\begin{displaymath}
\xymatrix{ X \ar[r]^(0.4)\eta \ar[dr]_f  &  dd^\ast X \ar[d]^{d(\tilde{f})} \\
           &  dY }
\end{displaymath}
is a diagonal weak equivalence since $\eta$ is a weak equivalence. Conversely,
suppose $\tilde{f}: d^\ast X \to Y$ is a diagonal weak equivalence. Then
$d(\tilde{f})$ is a weak equivalence so the composite $f = d(\tilde{f})\eta$ is a
weak equivalence, as was to be shown.
\end{proof}

Let $i^{\prime\prime\prime}$ denote the inclusion functor
\begin{displaymath}
i^{\prime\prime\prime}: H_d(x, z) \hra H(x, z).
\end{displaymath}
Again, $H_d(x, z)$ is \emph{not} a priori a full subcategory of $H(x, z)$. 

\begin{lem}
For any two pointed simplicial sheaves $x$ and $z$, the induced map
$\pi_0(i^{\prime\prime\prime})$ is injective.
\end{lem}
\begin{proof}
Suppose $m: (f, p) \to (g, q)$ is any morphism of $H(x, z)$. Then
$x \times z = d(x \times z)$ so that the map $(f, p): u \to x \times z$ uniquely
factors as
\begin{displaymath}
\xymatrix{  &  dd^\ast(u) \ar[d]^{d(\tilde{f}, \tilde{p})} \\
           u \ar[ur]^\eta \ar[r]_(0.4){(f, p)} &  x \times z }
\end{displaymath}
for some maps $\tilde{f}$, $\tilde{p}$ from $d^\ast(u)$ to $x$, $z$ by adjointness,
and similarly $(g, q): v \to x \times z$ factors as $\eta$ followed by a uniquely
determined pair $(\tilde{g}, \tilde{q})$ from $d^\ast(v)$ to $x \times z$. There is
then a commutative diagram
\begin{displaymath}
\xymatrix{ u \ar[r]^(0.4)\eta \ar[dd]_m  &  dd^\ast(u)
           \ar[dr]^{d(\tilde{f}, \tilde{p})} \ar[dd]^{dd^\ast(m)} &  \\
           &  &  x \times z \\
           v \ar[r]^(0.4)\eta  &  dd^\ast(v) \ar[ur]_{d(\tilde{g}, \tilde{q})} & }
\end{displaymath}
where both maps $\eta$ and the map $m$ are local weak equivalences (for $\eta$ use
Lemma \ref{unitmap} in sections), so $dd^\ast(m)$ is a local weak equivalence.
Further, $d(\tilde{f})$ is a local weak equivalence since $f$ and $\eta$ are local
weak equivalences, so $\tilde{f}$ is a diagonal local weak equivalence, and
similarly for $\tilde{g}$. The zigzag
\begin{displaymath}
(f, p) \xrightarrow{\eta} (d(\tilde{f}), d(\tilde{p})) \xrightarrow{dd^\ast(m)}
(d(\tilde{g}), d(\tilde{q})) \xleftarrow{\eta} (g, q)
\end{displaymath}
in $H(x, z)$ shows that the original map $m$ is in the same path component as
$dd^\ast(m)$, thus any morphism $m$ of objects in $H(x, z)$ naturally lifts to a
morphism $dd^\ast(m)$ in $H_d(x, z)$. It follows that any zigzag of maps in
$H(x, z)$ naturally lifts to a zigzag of maps in $H_d(x, z)$, so the result
follows.
\end{proof}

\begin{cor} \label{iprime}
Suppose $x$ and $z$ are two pointed simplicial sheaves as above such that $z$ is
locally fibrant. Then the induced maps
\begin{eqnarray*}
\pi_0(i^{\prime\prime\prime}): \pi_0 H_d(x, z)  \to  \pi_0 H(x, z)\qquad\mathrm{and}
\end{eqnarray*}
\begin{eqnarray*}
\pi_0(i^{\prime\prime}): \pi_0 H_{d-\hyp}(x, z) \to  \pi_0 H_{\hyp}(x, z)
\end{eqnarray*}
are bijections.
\end{cor}
\begin{proof}
The map $\pi_0(i^{\prime\prime\prime})$ is injective by the previous Lemma. Consider
the commutative square
\begin{displaymath}
\xymatrix{ \pi_0 H_{d-\hyp}(x, z) \ar[r]^(0.55){\pi_0(i^\prime)}
              \ar[d]_{\pi_0(i^{\prime\prime})}  &
              \pi_0 H_d(x, z) \ar[d]^{\pi_0(i^{\prime\prime\prime})}  \\
           \pi_0 H_{\hyp}(x, z) \ar[r]^{\pi_0(j)}   &  \pi_0 H(x, z) }
\end{displaymath}
induced by the corresponding inclusions. The top and bottom maps are bijections and
the lefthand vertical map is surjective, so $\pi_0(i^{\prime\prime\prime})$ is
surjective, hence bijective by the previous Lemma. But then the composite
$\pi_0(i^{\prime\prime\prime})\pi_0(i^\prime)$ is bijective so
$\pi_0(i^{\prime\prime})$ must also be injective, hence bijective.
\end{proof}

Recall as above that $H_{\bihyp}(x, y)$ is a full subcategory of $H_d(x, y)$ for
any fixed choice of pointed simplicial sheaves $x$ and $y$. Letting
\begin{displaymath}
i: H_{\bihyp}(x, y) \hra H_d(x, y)
\end{displaymath}
denote the inclusion functor, one therefore knows that the induced map $\pi_0(i)$
on path components is injective. To move towards bijectivity one begins with an
analogue of Corollary \ref{diagequiv}:

\begin{lem} \label{bkequiv}
If $f: X \to Y$ is a weak equivalence of fibrant simplicial sets then the induced
map $d_\ast(f): d_\ast X \to d_\ast Y$ is a degreewise weak equivalence of
bisimplicial sets.
\end{lem}
\begin{proof}
By Brown's factorization lemma it suffices to show that $d_\ast$ sends trivial
fibrations of simplicial sets to degreewise trivial fibrations of bisimplicial
sets. The induced map $d_\ast(f)$ is a degreewise trivial fibration if and only
if it has the right lifting property with respect to all Bousfield-Kan
cofibrations $c: A \to B$. By adjunction such lifting problems correspond to
lifting problems
\begin{displaymath}
\xymatrix{ d(A) \ar[r] \ar[d]_{d(c)} & X \ar[d]^f \\
           d(B) \ar[r] \ar@{.>}[ur]  & Y } 
\end{displaymath}
These all have solutions since $c$ is in particular a pointwise cofibration, hence
$d(c)$ is a cofibration and the lift exists since $f$ was a trivial fibration by
assumption.
\end{proof}

Here is the local version:
\begin{lem} \label{locbkequiv}
If $\beta: z \to z^\prime$ is a pointed local weak equivalence of pointed locally
fibrant simplicial sheaves then $d_\ast(\beta)$ is a pointed degreewise local weak
equivalence of bisimplicial sheaves.
\end{lem}
\begin{proof}
Fix a Boolean localization $p: \Shv{\mc{B}} \to \Shv{\mc{C}}$. Then $p^\ast(\beta)$
is a sectionwise weak equivalence of sectionwise fibrant simplicial sheaves on
$\mc{B}$ so $d_\ast(p^\ast(\beta))$ is a degreewise weak equivalence in each
section by Lemma \ref{bkequiv}. By the argument of Lemma \ref{locdequiv} it follows
that $p^\ast(d_\ast(\beta))$ is also a degreewise weak equivalence in each section,
or equivalently a sectionwise weak equivalence in each degree, so $d_\ast(\beta)$
is a local weak equivalence in each degree. The map $d_\ast(\beta)$ is automatically
pointed so the result follows.
\end{proof}

\begin{lem} \label{bfiblem}
If $f: x \to y$ is a pointed local fibration of simplicial sheaves, then $d_\ast(f)$
is a pointed degreewise local fibration of bisimplicial sheaves.
\end{lem}
\begin{proof}
Fix a Boolean localization $p: \Shv{\mc{B}} \to \Shv{\mc{C}}$. Then $p^\ast(f)$
is a sectionwise fibration of simplicial sheaves on $\mc{B}$. An adjunction
argument (starting from the fact that the diagonal of a degreewise trivial
cofibration is a trivial cofibration) shows that $d_\ast(p^\ast(f))$ is a
degreewise fibration in each section, or alternatively a sectionwise fibration in
each degree. Thus $p^\ast(d_\ast(f))$ is a sectionwise fibration in each
degree so $d_\ast(f)$ is a trivial fibration in each degree, as was to be shown.
\end{proof}

One has the following analogue of Lemma \ref{lwed}:
\begin{lem} \label{blwed}
If $\beta: z \to z^\prime$ is a local weak equivalence of pointed locally fibrant
simplicial sheaves then the induced functor
\begin{displaymath}
\beta_\ast: BH_{\bihyp}(x, z) \to BH_{\bihyp}(x, z^\prime)
\end{displaymath}
is a homotopy equivalence.
\end{lem}
\begin{proof}
The proof is analogous to that of Lemma \ref{lwed}: the functor $\beta_\ast$ is
defined by composition with $\beta$ and the induced map $d_\ast(\beta)$ is a
degreewise local weak equivalence by Lemma \ref{locbkequiv}. Supposing one has a
``cocycle'' of the form
\begin{displaymath}
x \xleftarrow{f} u \xrightarrow{\tilde{p}} d_\ast(z^\prime)
\end{displaymath}
where $f$ is a bisimplicial hypercover, one factors the map $(f, \tilde{p})$
as a sectionwise degreewise weak equivalence $c$ followed by a sectionwise
degreewise fibration $(p_x, \tilde{p}_{z^\prime})$. The map $p_x$ is a degreewise
local weak equivalence since $c$ and $f$ are, and is a degreewise local fibration
since it equals the composite
\begin{displaymath}
w \xrightarrow{(p_x, \tilde{p}_{z^\prime})} x \times d_\ast(z^\prime)
   \xrightarrow{\pr_L} x
\end{displaymath}
where the former map is a sectionwise fibration in each degree and the latter map
is a degreewise local fibration by Lemma \ref{bfiblem}; thus $p_x$ is again a
bisimplicial hypercover. The map $1_x \times d_\ast(\beta)$ is a degreewise local
weak equivalence so its pullback $(1_x \times d_\ast(\beta))_\ast$ along
$(p_x, \tilde{p}_{z^\prime})$ is also a degreewise local weak equivalence
by a Boolean localization argument. The other pullback map
$(p_x^\ast, \tilde{p}_z^\ast)$ is a sectionwise degreewise fibration since
$(p_x, \tilde{p}_{z^\prime})$ was, so $p_x^\ast$ is also a degreewise local
fibration. The map $p_x^\ast$ is also a degreewise local weak equivalence since
$p_x$, $1_x \times d_\ast(\beta)$, and $(1_x \times d_\ast(\beta))_\ast$ are, so
$p_x^\ast$ is also a bisimplicial hypercover. This construction determines the
functor $\tilde{\psi}$, and the remainder of the argument follows the proof of
Lemma \ref{lwed} verbatim.
\end{proof}

\begin{lem} \label{bhlem}
Suppose $x$ and $y$ are pointed simplicial sheaves on a pointed small Grothendieck
site $\mc{C}$ with $y$ locally fibrant. Then the map
\begin{displaymath}
\pi_0(i): \pi_0 H_{\bihyp}(x, y) \to \pi_0 H_d(x, y)
\end{displaymath}
induced by inclusion is a bijection.
\end{lem}
\begin{proof}
First assume that $y$ is globally fibrant. Starting with any ``cocycle''
\begin{displaymath}
x \xleftarrow{f} u \xrightarrow{\tilde{p}} d_\ast(y)
\end{displaymath}
corresponding to an object of $H_d(x, y)$, factor $f$ as a sectionwise degreewise
cofibration $c$ followed by a sectionwise degreewise trivial fibration $h$.
Then $h$ is a bisimplicial hypercover and one has an induced map of cocycles
\begin{displaymath}
\xymatrix{  &  d(u) \ar[dl]_{d(f)} \ar[dr]^p \ar[dd]^{d(c)}  &  \\
           x  &  &  y \\
            &  d(v) \ar[ul]^{d(h)} \ar@{.>}[ur] & }
\end{displaymath}
where $d(f)$ and $d(h)$ are local weak equivalences so $d(c)$ is a local weak
equivalence as well as a cofibration and the lift therefore exists since $y$ was
globally fibrant by assumption.

More generally, suppose $y$ is locally fibrant and fix a pointed globally fibrant
replacement $\beta: y \to z$ for $y$. Then there is a commutative square
\begin{displaymath}
\xymatrix{ \pi_0 H_{\bihyp}(x, y) \ar[r]^(0.55){\pi_0(i)} \ar[d]_{\beta_\ast} &
           \pi_0 H_d(x, y) \ar[d]^{\beta_\ast} \\
           \pi_0 H_{\bihyp}(x, z) \ar[r]^(0.55){\pi_0(i)}  &  \pi_0 H_d(x, z) }
\end{displaymath}
where both vertical maps $\beta_\ast$ are induced by composition with $\beta$
so they are bijections by Lemmas \ref{blwed} and \ref{lwed}, and the bottom
map is a bijection by the previous paragraph so the top map is a bijection as
well, as was to be shown.
\end{proof}

The results above may be summarized as follows:
\begin{thm} \label{thmbijs}
Suppose $x$ and $y$ are pointed simplicial sheaves on a pointed small Grothendieck
site $\mc{C}$ where $y$ is locally fibrant. Then there are canonical bijections
\begin{displaymath}
\pi_0 H_{\bihyp}(x, y) \cong \pi_0 H_{d-\hyp}(x, y) \cong \pi_0 H_d(x, y) \cong
\pi_0 H_{\hyp}(x, y) \cong \pi_0 H(x, y) \cong [x, y].
\end{displaymath}
\end{thm}
\begin{proof}
The latter bijection is a consequence of Theorem $1$ of \cite{Jardine4}. The
remaining bijections have already been established.
\end{proof}

\begin{cor}
With the hypotheses of Theorem \ref{thmbijs}, there are canonical bijections
\begin{displaymath}
\pi_0 H_{\bihyp}(x, y) \xrightarrow{\cong} \pi_0 H_{\bihyp}(x^\prime, y^\prime)
\end{displaymath}
induced by any two local weak equivalences
$\alpha: x \to x^\prime$, $\beta: y \to y^\prime$ of pointed locally fibrant
simplicial sheaves on $\mc{C}$.
\end{cor}
\begin{proof}
Use the analogous property for $H(x, y)$, proven in Lemma $1$ of \cite{Jardine4}.
\end{proof}

\section{Applications to \'etale homotopy theory}

Inspired by \cite{Jardine13}, consider the cocycle category
$H_{\bihyp}^{h^\prime}(x, y)$ whose objects are cocycles of the form
\begin{displaymath}
x \xleftarrow{d(f)} d(u) \xrightarrow{[p]} y
\end{displaymath}
for pointed bisimplicial hypercovers $u$ of $x$ and whose morphisms are diagrams
\begin{displaymath}
\xymatrix@=20pt{  &   d(u) \ar[dl]_{d(f)} \ar[dr]^{[p]} \ar[dd]^{d([m])} &  \\
           x  &   &   y \\
	    &   d(u^\prime) \ar[ul]^{d(f^\prime)} \ar[ur]_{[p^\prime]}  &  }
\end{displaymath}
where square brackets indicate simplicial homotopy classes of maps, the middle maps
are induced by fibrewise simplicial homotopy classes $[m]: u \to u^\prime$ of maps
of pointed bisimplicial hypercovers of $x$, and $[p^\prime][d(m)] = [p]$ as
pointed simplicial homotopy classes. The set of path components
$\pi_0 H_{\bihyp}^{h^\prime}(x, y)$ is given by the colimit
\begin{displaymath}
\colim_{d(f):\, d(u) \to x}{\pi(d(u), y)}
\end{displaymath}
whose index category is that of pointed bisimplicial hypercovers of $x$ and pointed
fibrewise simplicial homotopy classes of maps between them. The functor
$\omega: H_{\bihyp}(x, y) \to H_{\bihyp}^{h^\prime}(x, y)$ defined on objects by
$(d(f), p) \mapsto (d(f), [p])$ is obviously surjective on path components.

Introduce another cocycle category $H_{\bihyp}^h(x, y)$ whose objects are of the
form
\begin{displaymath}
x \xleftarrow{[d(f)]} d(u) \xrightarrow{[p]} y
\end{displaymath}
and whose morphisms are commutative diagrams
\begin{displaymath}
\xymatrix@=20pt{  &   d(u) \ar[dl]_{[d(f)]} \ar[dr]^{[p]} \ar[dd]^{[d(m)]} &  \\
           x  &   &   y \\
	    &   d(u^\prime) \ar[ul]^{[d(f^\prime)]} \ar[ur]_{[p^\prime]}  &  }
\end{displaymath}
where the middle maps are induced by maps $m: u \to u^\prime$ of pointed
bisimplicial hypercovers of $x$. One readily verifies that the maps
$(d(f), [p]) \mapsto ([d(f)], [p])$ on objects and $d([m]) \mapsto [d(m)]$ on
morphisms determine a functor $\omega^\prime: H_{\bihyp}^{h^\prime}(x, y)
\to H_{\bihyp}^h(x, y)$ which again is obviously surjective on path components.

To relate these observations back to a homotopy category, suppose that the point
$x: \Set \to \Shv{\mc{C}}$ comes from an object $\Omega$ of $\mc{C}$ representing
a sheaf, in the sense that the inverse image functor $x^\ast$ is given by a
composite
\begin{displaymath}
\Shv{\mc{C}} \xrightarrow{?_{|\Omega}} \Shv{\mc{C}/\Omega} \xrightarrow{\Gamma_\ast} \Set
\end{displaymath}
defined by first restricting to the site $\mc{C}/\Omega$ and then taking global
sections. This is exactly the situation in \'etale homotopy theory when one works
on a ``big'' \'etale site containing the separably closed field
$\Omega \deq \Spec{\Omega}$ which is used to give the geometric point $x$ of the
base scheme or DM stack $S$. A \emph{pointed} (bi)simplicial sheaf $(X, z)$ on such
a site $(\mc{C}, x)$ then corresponds exactly to a section
\begin{displaymath}
\Omega \xrightarrow{z} X
\end{displaymath}
where $\Omega = \Omega \xrightarrow{x} S$ is the object of $\mc{C}$ corresponding
to the point $x$, and a pointed map $(X, z) \to (Y, z^\prime)$ corresponds exactly
to a map $X \to Y$ respecting the sections $z$ and $z^\prime$. By general nonsense
there is a closed model structure on the category $\Omega/\sShv{\mc{C}}$ of pointed
simplicial sheaves where the fibrations (resp. cofibrations, resp. weak
equivalences) are those maps $(X, z) \to (Y, z^\prime)$ under $\Omega$ such that
the underlying maps $X \to Y$ are fibrations (resp. cofibrations, resp. weak
equivalences). Lemma $1$ of \cite{Jardine13} then says that the canonical map
\begin{displaymath}
\pi_0 H(x, y) \to [x, y]
\end{displaymath}
defined by sending any pointed cocycle $(f, g)$ to the composite $gf^{-1}$ is a
bijection, where $[x, y]$ denotes morphisms in $\Ho{\Omega/\sShv{\mc{C}}}$.

\begin{lem} \label{lotsobijs}
Suppose $(\mc{C}, x)$ is a pointed small Grothendieck site such that the point
$x: \Set \to \Shv{\mc{C}}$ is determined by an object $\Omega$ of $\mc{C}$ as above
which represents a sheaf (or suppose $\Omega = \emptyset$). Then for any two pointed
simplicial sheaves $x$ and $y$ on $\mc{C}$ with $y$ locally fibrant there are
canonical bijections
\begin{displaymath}
\pi_0 H_{\bihyp}(x, y) \xrightarrow{\pi_0(\omega)}
\pi_0 H_{\bihyp}^{h^\prime}(x, y) \xrightarrow{\pi_0(\omega^\prime)}
\pi_0 H_{\bihyp}^h(x, y) \xrightarrow{p} [x, y].
\end{displaymath}
\end{lem}
\begin{proof}
The displayed composite $c = p\pi_0(\omega^\prime)\pi_0(\omega)$ factors as
\begin{displaymath}
\pi_0 H_{\bihyp}(x, y) \xrightarrow{\pi_0(j^\prime)} \pi_0 H(x, y)
\xrightarrow{\cong} [x, y]
\end{displaymath}
where $j^\prime: H_{\bihyp}(x, y) \to H(x, y)$ is the inclusion functor and the
second arrow is the canonical map sending any cocycle $(f, g)$ to the composite
$gf^{-1}$ in the homotopy category. As $j^\prime = i^{\prime\prime\prime}i$,
$\pi_0(j^\prime)$ is a bijection by Corollary \ref{iprime} and Lemma \ref{bhlem}
so that $c$ is also a bijection. This implies that the canonical map $p$ is a
surjection and that $\pi_0(\omega)$ is an injection, hence a bijection. But
then $p\pi_0(\omega^\prime)$ is also a bijection, so $\pi_0(\omega^\prime)$ is an
injection, hence a bijection, and thus $p$ is also a bijection, as was to be
shown.
\end{proof}

One may then establish statements such as
\begin{prop} \label{verd1}
Suppose $\mc{C}$ is a small Grothendieck site, $H$ a sheaf of groups on $\mc{C}$,
$n \geq 0$ if $H$ is abelian or $0 \leq n \leq 1$ otherwise, and $X$ a simplicial
sheaf on $\mc{C}$. Then there are canonical bijections
\begin{displaymath}
H^n(X, H) \cong \colim_{p:\, d(U) \to X}{H^n(dU, H)}
\end{displaymath}
where the colimit is indexed over the category of bisimplicial hypercovers of $X$
and fibrewise simplicial homotopy classes of maps between them. If $X$ is
representable then these extend to bijections
\begin{displaymath}
H^n(\mc{C}/X, H_{|X}) \cong \colim_{p:\, d(U) \to X}{H^n(dU, H)}
\end{displaymath}
where $H_{|X}$ is defined by $H_{|?}(U \to X_n) \deq H(U)$.
\end{prop}
\begin{proof}
There is a series of identifications
\begin{eqnarray*}
H^n(X, H) & \deq  & [X, K(H, n)] \\
	  & \cong & \colim_{p:\, d(U) \to X}{\pi(dU, GK(H, n))} \\
	  & \cong & \colim_{p:\, d(U) \to X}{[dU, GK(H, n)]} \\
	  &  =    & \colim_{p:\, d(U) \to X}{H^n(dU, H)}
\end{eqnarray*}
where $GK(H, n)$ is a globally fibrant model for $K(H, n)$, the first isomorphism
is by Lemma \ref{lotsobijs} and the identification of
$\pi_0 H_{\bihyp}^{h^\prime}(X, GK(H, n))$, and the final identification is by
definition of cohomology of $dU$ with coefficients in $H$. The latter statement
follows from the proof of Theorem $3.10$ of \cite{Jardine3}.
\end{proof}
One may compare this with the statement ($3.8$, \cite{Friedlander1}), keeping in
mind that the bisimplicial hypercovers $U \to X$ here are \emph{not} assumed to be
representable.

\subsection{\'Etale homotopy types from bisimplicial hypercovers}

To return to the \'etale homotopy type $T(X, z)$ defined in \S\ref{secthyp}, one
is motivated by the above results to consider the diagram $T_b(X, z)$ for any
pointed connected simplicial sheaf $x = (X, z)$ on a pointed locally connected
small Grothendieck site $\mc{C}$ given by the simplicial sets $\ccomp d(u)$ for
pointed bisimplicial hypercovers $u \to x$, with maps induced by the pointed
fibrewise simplicial homotopy classes of maps between the bisimplicial hypercovers
over $x$. That this determines a pro-object in $\Ho{\sSet_\ast}$ is a consequence
of the functoriality of the construction of equalizers for simplicial homotopy
classes of maps after noting that the sheaves $X_m$ are globally, hence locally,
fibrant.

The following result, which is the main point of this work, gives the relationship
between $T(X, z)$ and $T_b(X, z)$ for locally fibrant $X$:
\begin{thm} \label{mainthm}
Suppose $\mc{C}$ is a pointed locally connected small Grothendieck site such that
the point $\Set \to \Shv{\mc{C}}$ is determined by some object $\Omega$ representing
a sheaf, and $x = (X, z)$ a pointed connected locally fibrant simplicial sheaf
on $\mc{C}$. Then the pro-object $T(X, z)$ of $\Ho{\sSet_\ast}$ is canonically
pro-isomorphic to the pro-object $T_b(X, z)$, and similarly for the unpointed
variants in $\Ho{\sSet}$. Furthermore, these pro-isomorphisms
$T(X, z) \cong T_b(X, z)$ are functorial in $(X, z)$, and similarly for the
unpointed case.
\end{thm}
\begin{proof}
For the first part it suffices to give a canonical natural isomorphism between the
functors that these pro-objects pro-represent. On the one hand one has canonical
isomorphisms
\begin{eqnarray*}
[T_b(X, z), y] & \deq & \colim_{p:\, d(u) \to x}{[\ccomp d(u), y]} \\
  & \cong & \colim_{p:\, d(u) \to x}{\pi(\ccomp d(u), y)} \\
  & \cong & \colim_{p:\, d(u) \to x}{\pi(d(u), \Gamma^\ast y)} \\
  & \cong & \pi_0 H(x, \Gamma^\ast y) \\
  & \cong & [x, \Gamma^\ast y]
\end{eqnarray*}
natural in $y$ by Theorem \ref{thmbijs} and Lemma \ref{lotsobijs}, where one may
assume $y$ is fibrant by applying the $\Exinfn$ functor. By a similar calculation
one has $[T(X, z), y] \cong [x, \Gamma^\ast y]$ naturally in $y$. Therefore there
are canonical natural bijections $[T_b(X, z), y] \cong [T(X, z), y]$, so the
functors $[T_b(X, z), -]$ and $[T(X, z), -]$ are canonically naturally isomorphic 
and thus the representing pro-objects are canonically pro-isomorphic, as was to be
shown. Forgetting the points yields the analogous statement for the unpointed case.

For functoriality, observe that the cocycle categories
$H_{\bihyp}(x, \Gamma^\ast y)$ and $H_{\hyp}(x, \Gamma^\ast y)$ determine
contravariant functors $f^\ast$ in $x$ by sending any cocycle of the form
\begin{displaymath}
x \xleftarrow{d(p)} d(u) \xrightarrow{q} \Gamma^\ast y
\end{displaymath}
to the cocycle
$(d(p_\ast), q \cdot d(\pr_u)): d(u^\prime) \to x^\prime \times \Gamma^\ast y$
induced by pulling back $p$ to $p_\ast$ along any pointed map $f: x^\prime \to x$
of pointed simplicial sheaves, and similarly for $H_{\hyp}$. Consider the diagram
\begin{displaymath}
\xymatrix@=20pt{ [T_b(x), y] \ar[dd]_{\cong} \ar[dr]^\cong \ar[rr]^{f^\ast}
                   &  & [T_b(x^\prime), y] \ar[dr]^{\cong} \ar'[d][dd]^(0.4){\cong}
                   &  \\
                   & \pi_0 H_{\bihyp}(x, \Gamma^\ast y) \ar[rr] \ar[dl]_{\cong}
                   &  &  \pi_0 H_{\bihyp}(x^\prime, \Gamma^\ast y)
	           \ar[dl]^{\cong} \\
	   [x, \Gamma^\ast y] \ar[rr]^{f^\ast}  & &  [x^\prime, \Gamma^\ast y] }
\end{displaymath}
where the two triangles on either end are canonical factorizations of the canonical
natural isomorphisms above. The back square commutes if the top and the front
squares commute. Any element of the set
\begin{displaymath}
\colim_{p:\, d(u) \to x}{\pi(d(u), \Gamma^\ast y)}
\end{displaymath}
is represented by some pair $(p: u \to x, [q]: d(u) \to \Gamma^\ast y)$ where $p$
is a pointed bisimplicial hypercover, and $f^\ast$ (precomposition with $f$) on
$[T_b(x), y]$ induces a map $f^\ast$ sending such a pair to the element of
\begin{displaymath}
\colim_{p:\, d(u) \to x^\prime}{\pi(d(u), \Gamma^\ast y)}
\end{displaymath}
represented by the pair $(p_\ast: u^\prime \to x^\prime, [q \cdot d(\pr_u)])$;
clearly this is compatible with the functor $\pi_0 H_{\bihyp}(-, \Gamma^\ast y)$
so the top square commutes. Chasing the cocycle $(d(p), q)$ around the front square
gives
\begin{displaymath}
(d(p), q) \mapsto [qd(p)^{-1}] \mapsto [qd(p)^{-1}f]
\end{displaymath}
on the left side and
\begin{displaymath}
(d(p), q) \mapsto (d(p_\ast), q \cdot d(\pr_u)) \mapsto [qd(\pr_u)d(p_\ast)^{-1}]
\end{displaymath}
on the other, but these are equal since $d(p)d(\pr_u) = fd(p_\ast)$ by definition
of $p_\ast$. Thus the front square also commutes, so the back square commutes. A
similar computation with $\pi_0 H_{\hyp}(-, \Gamma^\ast y)$ shows that the analogous
square commutes, giving the functoriality in $x$. The unpointed case follows by
the same argument.
\end{proof}

To move towards a statement about the invariance of $T_b(X, z)$ under local weak
equivalences, there is
\begin{lem} \label{exinf}
Under the same assumptions on $\mc{C}$ as in Theorem \ref{mainthm}, the pro-object
$T_b(X, z)$ associated to any pointed connected simplicial sheaf $x = (X, z)$ on
$\mc{C}$ is canonically pro-isomorphic to the pro-object $T_b(\Exinf(X), z^\prime)$
in $\Ho{\sSet_\ast}$ where $z^\prime$ is induced by the canonical sectionwise weak
equivalence $\eta_X: X \to \Exinf(X)$, and similarly for the unpointed variants in
$\Ho{\sSet}$. Furthermore, these pro-isomorphisms
$T_b(X, z) \cong T_b(\Exinf(X), z^\prime)$ are functorial in $(X, z)$, and
similarly for the unpointed case.
\end{lem}
\begin{proof}
For the first statement, observe that there are canonical bijections
\begin{eqnarray*}
[T_b(\Exinf(X), z^\prime), y]
 & \deq & \colim_{p:\,d(u) \to \Exinf(X)}{[\ccomp d(u), y]} \\
 & \cong & \pi_0 H^{h^\prime}_{\bihyp}(\Exinf(X, z^\prime), \Gamma^\ast y) \\
 & \cong & \pi_0 H(\Exinf(X, z^\prime), \Gamma^\ast y) \\
 & \cong & [\Exinf(X, z^\prime), \Gamma^\ast y] \\
 & \cong & [x, \Gamma^\ast y] \\
 & \cong & [T_b(X, z), y]
\end{eqnarray*}
natural in $y$ for any fibrant simplicial set $y$ by Lemma \ref{lotsobijs},
Theorem \ref{thmbijs}, Lemma $1$ of \cite{Jardine4}, and Theorem $1$ of
\cite{Jardine4}. These naturally extend to canonical bijections for arbitrary
$y$ via $\Exinfn$. After the proof of Theorem \ref{mainthm}, establishing
functoriality in $x$ reduces to verifying that the square
\begin{displaymath}
\xymatrix{ \pi_0 H_{\bihyp}(\Exinf(x), \Gamma^\ast y) \ar[r]^{f^\ast} \ar[d]_{\cong}  &  \pi_0 H_{\bihyp}(\Exinf(x^\prime),
             \Gamma^\ast y) \ar[d]^{\cong} \\
	    [x, \Gamma^\ast y] \ar[r]^{f^\ast}  &  [x^\prime, \Gamma^\ast y] }
\end{displaymath}
commutes. Chasing a cocycle $(d(p), q): d(u) \to \Exinf(x) \times \Gamma^\ast y$
around this square gives
\begin{displaymath}
(d(p), q) \mapsto [q d(p)^{-1} \eta_x] \mapsto [q d(p)^{-1} \eta_x f]
\end{displaymath}
on the left side and
\begin{displaymath}
(d(p), q) \mapsto (d(p_\ast), q \cdot d(\pr_u))
   \mapsto [q \cdot d(\pr_u) d(p_\ast)^{-1} \eta_{x^\prime}]
\end{displaymath}
on the right side, but these are equal by the definition of $p_\ast$ and the
fact that $\Exinfn$ is a functor.
\end{proof}

The key point here is that $T_b(X, z)$ exists without the requirement that $X$ be
locally fibrant, and the cocycle category techniques of \cite{Jardine4} make no
fibrancy assumptions on $X$. Here is the major consequence:

\begin{cor} \label{maincor}
With the same hypotheses as Theorem \ref{mainthm} on $\mc{C}$, suppose $x = (X, z)$
and $y = (Y, y)$ are pointed connected simplicial sheaves on $\mc{C}$, and
$f: (Y, y) \to (X, z)$ a pointed local weak equivalence. Then the strict morphism
of pro-objects $f_\ast: T_b(Y, y) \to T_b(X, z)$ induced from $f$ by pulling back
bisimplicial hypercovers along $f$ is a pro-isomorphism, and similarly for the
unpointed setting.
\end{cor}
\begin{proof}
The pointed map $f$ induces a pointed map $\eta_f: \Exinf(Y) \to \Exinf(X)$ by
functoriality of $\Exinfn$. As $\Exinf(X)$ and $\Exinf(Y)$ are locally fibrant and
connected there are canonical natural isomorphisms
\begin{displaymath}
[T_b(\Exinf(X)), -] \cong [T(\Exinf(X)), -]
\end{displaymath}
and similarly for $\Exinf(Y)$ by Theorem \ref{mainthm}. Then there is a commutative
diagram
\begin{displaymath}
\xymatrix{ [T_b(X), -] \ar[d]^{f^\ast}  &  [T_b(\Exinf(X)), -] \ar[l]_(0.55){\eta_x^\ast} \ar[r]^{\cong} \ar[d]^{f^\ast} &
              [T(\Exinf(X)), -] \ar[d]^{\eta_f^\ast} \\
           [T_b(Y), -]  &  [T_b(\Exinf(Y)), -] \ar[l]^(0.55){\eta_y^\ast} \ar[r]_{\cong} &
	      [T(\Exinf(Y)), -] }
\end{displaymath}
where $\eta_f^\ast$ is a natural isomorphism by Proposition \ref{invwe}. The
righthand square is commutative by the functoriality in Theorem \ref{mainthm} so
the middle $f^\ast$ is also a natural isomorphism. The lefthand square commutes
by the functoriality in Lemma \ref{exinf} and the maps $\eta_x^\ast$ and
$\eta_y^\ast$ are natural isomorphisms since $\eta_x$ and $\eta_y$ are local weak
equivalences, so the map $f^\ast$ on the left is a natural isomorphism. This map
$f^\ast$ was induced by precomposition with the map
$f_\ast: T_b(Y, y) \to T_b(X, z)$ induced by $f$ itself, so $f_\ast$ is a
pro-isomorphism, as was to be shown.
\end{proof}

\bibliographystyle{plain}
\bibliography{bsimp}

\end{document}